\documentclass[reqno]{amsart}

\numberwithin{equation}{section} 
\usepackage{amsmath,amsfonts,amssymb,amsthm,latexsym}
\usepackage{enumerate}
\usepackage{cancel}
\usepackage{mathrsfs}
\usepackage{stackrel}

\usepackage{cases}
\usepackage[square,comma,numbers,sort&compress]{natbib}
\usepackage{verbatim}

\usepackage[colorlinks=true]{hyperref}
\hypersetup{urlcolor=blue, citecolor=red}

\newcounter{mnote}
\theoremstyle{plain}
\newtheorem{theorem}{Theorem}[section]
\newtheorem{proposition}[theorem]{Proposition}
\newtheorem{lemma}[theorem]{Lemma}

\theoremstyle{definition}
\newtheorem{definition}[theorem]{Definition}

\theoremstyle{remark}

\newcommand{\od}[2]{\frac{d #1}{d #2}}

\newcommand{\norm}[1]{\left\lVert#1\right\rVert}

\newcommand*{\ang}[1]{\left\langle #1 \right\rangle}

\def\aa{\alpha}

\def\ll{\lambda}
\def\Dd{\Delta}
\def\Gg{\Gamma}
\def\Om{\Omega}

\def\pp{\partial}

\begin{document}
\title[On the Charney Conjecture of Data Assimilation]{ On the Charney Conjecture of Data Assimilation Employing Temperature Measurements Alone:  The Paradigm of 3D Planetary Geostrophic Model}

\date{August 16, 2016}

%
\author{Aseel Farhat}
\address[Aseel Farhat]{Department of Mathematics\\
               University of Virginia\\
       Charlottesville, VA 22904, USA}
\email[Aseel Farhat]{af7py@virginia.edu}
\author{Evelyn Lunasin}
\address[Evelyn Lunasin]{Department of Mathematics\\
                United States Naval Academy\\
        Annapolis, MD, 21402 USA}
\email[Evelyn Lunasin]{lunasin@usna.edu}

\author{Edriss S. Titi}
\address[Edriss S. Titi]{Department of Mathematics, Texas A\&M University, 3368 TAMU,
 College Station, TX 77843-3368, USA.  {\bf ALSO},
  Department of Computer Science and Applied Mathematics, Weizmann Institute
  of Science, Rehovot 76100, Israel.} \email{titi@math.tamu.edu and
  edriss.titi@weizmann.ac.il}

\begin{abstract}
Analyzing the validity and success of a data assimilation algorithm when some state variable observations are not available is an important problem in meteorology and engineering. We present an improved data assimilation  algorithm for recovering the exact full reference solution (i.e.~the velocity and temperature) of the 3D Planetary Geostrophic model, at an exponential rate in time, by employing  coarse spatial mesh observations of the temperature alone. This provides, in the case of this paradigm, a rigorous justification to an earlier conjecture of Charney which states that temperature history of the atmosphere, for certain simple atmospheric models, determines all other state variables.\end{abstract}

\maketitle

Numerical models for geophysical process, such as the primitive-equation, require accurate initialization process in order to make accurate predictions. The initialization process depends on how acquired observations, such as temperature and velocity measurements, for example, are properly interpolated in space and time, to attempt to complete the information across the space-time domain, while maintaining the dynamical balance between these fields.  Direct observations of these fields may be readily available but often not uniformly distributed in space and are very sparse. In particular, the errors contained in the measurements and the model parameters, combined with the highly nonlinear property of the governing model equations, makes basic interpolation not a good starting point when designing an initial condition even for short term prediction. Meteorologists have devised several diagnostic tests when designing accurate initialization procedures that minimizes the loss of information from the acquired data.  For example, the interpolated function in space and time must satisfy the conservation laws in the continuous model equations. Meteorologists also use different combinations of collected information about the state of the system to see which combinations will yield the system close to the collected data. In this process, they may also use other forms of collected measurements such as, the average temperature of a small region obtained through image processing of satellite observation. One must know how to properly make use of this information, in particular when the collected data is not one of the evolving state variable in the numerical model.

In the context of meteorology and atmospheric physics,  data assimilation algorithm  when some state variable observations are not available as an input, has been studied in \cite{Charney1969, Errico85, Hoke-Anthes1976, LorencTibaldi80}  for some simplified numerical forecast models.   Although nonlinear interactions between the scales of motion and the parameters in the system exist, it has been shown in several settings that if the dynamical model used as a source of apriori information captures the important properties of the system being modeled, then one can identify the full state of the system knowing only coarse observation from an partial data that is selected properly.  For example, the numerical experiment of Charney in  \cite{Charney1969}, confirms that wind and surface pressure can be determined from coarse mesh measurements of temperatures alone.   The numerical experiments  of Lorenc and Tibaldi in \cite{LorencTibaldi80} showed that frontal humidity fields can be determined from height and  wind data.  Their numerical experiments report on the influence of the distance between observations, i.e. the size of the grid points, the different combinations of the data being assimilated, and the effect of measurement errors in the accuracy of the initialization process.  The experiments of \cite{Charney1969} and \cite{ Hoke-Anthes1976} which use the nudging data assimilation scheme, for example, were also able to approximate the size of the relaxation time scale $\frac{1}{\mu}$ in order for the dynamics to adjust properly to the observations.  They have noted that if $\mu$ is set too small, the errors in the observations can get too large which results in  the nudging term infective to control the instabilities.   On the other hand, if $\mu$ is set too large, the dynamics will not have enough time to relax to the observation values.  In this work we will give estimates using rigorous analysis that properly balances the effect of the relaxation parameter and other physical parameters in order to get an accurate initial condition using a data assimilation algorithm for the 3D planetary geostrophic model. On a similar note, it is worthwhile to mention some recent results in \cite{Law2016}, where the authors have derived rigorous conditions for an ODE system that partial observations must have in order to control the inherent uncertainty due to chaos.  The authors in \cite{Law2016} presents a summary of 3DVAR  for continuous and discrete time observations that highlights some important connection between 3DVAR, nudging and direct insertion (which they call as synchronization filter€) data assimilation schemes arising from various limiting conditions.

In \cite{FLT_horizontal}  we proposed a data assimilation  algorithm for a two-dimensional B\'enard convection problem: two-dimensional Boussinesq system of a layer of incompressible fluid between two solid horizontal walls, with no-normal flow and stress free boundary condition on the walls, and fluid is heated from the bottom and cooled from the top. We incorporate the observables as a feedback (nudging) term in the evolution equation of the {\it horizontal} velocity alone. We show that under an appropriate choice of the nudging parameter and the size of the spatial coarse mesh observables, and under the assumption that the observed data is error free, the solution of the proposed algorithm converges at an exponential rate to the unique exact unknown reference solution of the original system, associated with the observed data on the horizontal component of the velocity.  For this system we conjecture that we may not able to show that incomplete historical data on the temperature alone can determine the full state of the system.   Recent numerical studies in \cite{Altaf} shows support of this conjecture.

On the other hand in \cite{FLT_porous} we show that for the B\'enard convection in porous media one only needs to use  discrete spatial-mesh measurements of the temperature to show that the  solution of the proposed algorithm converges at an exponential rate in time, to the unique exact unknown reference solution of the original system, associated with the observed finite dimensional projection of temperature data.

Charney's question in \cite{Charney1969} of whether temperature observations are enough to determine all the dynamical state of the system results, in many ways, motivated a series of our recent studies, for example, see \cite{ Altaf, FLT2015} and references therein.   In this work we introduce a  data assimilation algorithm for the 3D Planetary Geostrophic model that {\it requires observations of temperature only}. This gives a rigorous support to an earlier conjecture of Charney that temperature observations can determine the dynamical variables in the system for certain simple atmospheric models.


In this article we consider the following planetary geostrophic
viscous model for oceanic and atmosphere dynamics (see, e.g. \cite{PG_Cao_Titi}, \cite{Cao_Titi_Ziane},\cite{STW98}, \cite{SV97A}, \cite{SV97B}, \cite{Vallis})
\begin{subequations}\label{PG1}
\begin{align}
&\nabla p + f \vec{k} \times u +  L_1 u = 0,  \label{EQ-1}  \\
&\pp_z p  + T =0,    \label{EQ-2}  \\
&\nabla \cdot u +\pp_z w =0,   \label{EQ-3} \\
&\pp_t T + u \cdot \nabla T
+ w \pp_z T +L_2 T = Q.  \label{EQ-4}
\end{align}
For simplicity, we focus in the case of ocean dynamics and consider the above system in the domain
\[
\Om= M \times (-H, 0)\subset \mathbb{R}^3,
\]
where $M$ is a bounded smooth domain in $\mathbb{R}^2,$  or the square
$M=(0,1) \times (0,1).$  Here $u=(u_1, u_2)$, and $(u_1, u_2, w)$
is the velocity field, $T$ is the temperature,
and $p$ is the pressure. $f=f_0(\beta+y)$ is the Coriolis parameter,
and $Q$ is a heat  source.
The dissipation operators $L_1$ and $L_2$ are given by
\begin{align*}
& L_1 = -A_h \Dd -A_v \pp^2_z,    \\
& L_2 = -K_h \Dd -K_v \pp^2_z,
\end{align*}
where
$A_h$ and  $A_v$ are positive molecular viscosities, and
$K_h$ and $K_v$ are positive conductivity constants.
We set $\nabla p = (\pp_x p, \pp_y p), \nabla \cdot u  =
\pp_x u_1+\pp_y u_2$ and $\Dd = \pp_x^2 +\pp_y^2.$
We denote the different parts of the boundary of $\Om$ by:
\begin{align*}
&\Gg_u = \{ (x,y,z)  \in \Om : z=0 \},  \\
&\Gg_b = \{ (x,y,z)  \in \Om : z=-H \},  \\
&\Gg_s = \{ (x,y,z)  \in \Om : (x,y) \in \partial M  \}.
\end{align*}
We equip system (\ref{EQ-1})--(\ref{EQ-4}) with
the following boundary conditions -- with wind-driven stress on the top
surface and stress-free and non-flux on the side walls and bottom
(see, e.g., \cite{PJ87}, \cite{SR97},
\cite{SV97A}, \cite{SV97B}, \cite{SD96}):
\begin{align}
&\mbox{on } \Gg_u: A_v  \frac{\pp u }{\pp z} = \tau,  \;
w=0, \; -K_v \frac{\pp T }{\pp z} = \aa (T-T^*);
\label{B-1}\\
&\mbox{on } \Gg_b: \frac{\pp u }{\pp z} = 0, \; w=0, \;
 \frac{\pp T }{\pp z} = 0;
\label{B-2}\\
&\mbox{on } \Gg_s:
u \cdot \vec{n} =0, \; \frac{\pp v}{\pp \vec{n}}  \times \vec{n} =0,
\;
\frac{\pp T }{\pp \vec{n}} =0,   \label{B-3}
\end{align}
where $\tau (x,y)$ is the given wind stress, $\vec{n}$ is the normal
vector of $\Gg_s$,
$T^* (x,y)$ is typical temperature of the top (upper) surface, and
$\aa > 0$ is a positive constant.
Due to the boundary conditions (\ref{B-1})--(\ref{B-3}),
it is natural to assume that $T^*$ satisfies the compatibility
boundary  condition:
\begin{equation}
\frac{\pp T^* }{\pp \vec{n}} =0   \qquad
\mbox{on  } \pp M.
\label{COM}
\end{equation}
In addition, we supply system \eqref{EQ-1}-\eqref{COM} with the initial condition:
\begin{equation}
T(x,y,z,0) = T_0 (x,y,z).  \label{INIT}
\end{equation}
\end{subequations}

\subsection{Alternate Formulation}
Following  \cite{STW98}, we can derive an equivalent
formulation for the system
(\ref{EQ-1})--(\ref{INIT}). By integrating equation (\ref{EQ-3})
in the $z$ direction, we obtain
\begin{equation}
w(x,y,z,t) = w(x,y,-H,t) - \int_{-H}^z \nabla \cdot u(x,y, \xi,t) d\xi.
\label{WWWW}
\end{equation}
Since $w(x,y,z,t)=0$ at $z=-H, 0$ (see (\ref{B-1}) and (\ref{B-2})),
we have
\begin{equation}
w(x,y,z,t) =  - \int_{-H}^z \nabla \cdot u(x,y, \xi,t) d\xi
\label{DIV-1}
\end{equation}
and
\begin{equation}
\int_{-H}^0 \nabla \cdot u(x,y, \xi,t) d\xi
= \nabla \cdot  \int_{-H}^0 u(x,y, \xi,t) d\xi =0.
\label{DIV}
\end{equation}
By integrating equation (\ref{EQ-2}) with respect to $z$
we obtain
\begin{equation}
p(x,y,z,t) = - \int_{-H}^z T(x,y,\xi,t) d\xi + p_s(x,y,t),
\label{PPPP}
\end{equation}
where $p_s(x,y,t)$ is a free function (the bottom pressure) to be determined.
Moreover,  notice that by setting
\begin{equation}
T=T^*+\widetilde{T}.  \label{TTTT}
\end{equation}
we convert the boundary condition (\ref{B-1}) to be homogeneous,
namely, $\widetilde{T}$ satisfies the following homogeneous
boundary conditions:
\begin{equation}
\left. \frac{\pp \widetilde{T}}{\pp z }
\right|_{z=-H}= 0;\quad
\left. {\left( \frac{\pp \widetilde{T}}{\pp z}
+ \frac{\aa}{K_v} \widetilde{T} \right) }
\right|_{z= 0}= 0;\quad
\; \left. \frac{\pp \widetilde{T}}{\pp \vec{n}} \right|_{\Gg_s}= 0,
\label{HB}
\end{equation}
(notice that we have also used the compatibility condition
(\ref{COM})).
Based on all the above  we get the following new formulation for
system (\ref{EQ-1})--(\ref{INIT}):
\begin{subequations}\label{PG}
\begin{align}
&\nabla  \left[ p_s(x,y,t) - \int_{-H}^z \widetilde{T} (x,y,\xi,t) d\xi
-  (z+H) T^* (x,y,t) \right]+ f \vec{k} \times u +  L_1 u  = 0,  \label{TEQ-1}   \\
&\nabla \cdot \int_{-H}^0  u(x,y,z,t) \; dz = 0,    \label{TEQ-11}     \\
&\pp_t \widetilde{T} +L_2 \widetilde{T}
+ u \cdot \nabla \widetilde{T}
-\left( \nabla \cdot \int_{-H}^z  u(x,y,\xi,t) \,d\xi
\right)  \pp_z \widetilde{T}  +   u  \cdot \nabla T^* =Q^*,  \label{TEQ-2}   \\ 
&\left. \frac{\pp u }{\pp z} \right|_{z=0} = \tau, \;
\left. \frac{\pp u }{\pp z} \right|_{z=-H} = 0, \;
\left. u\cdot \vec{n} \right|_{\Gg_s} = 0, \;
\left. \frac{\pp u}{\pp \vec{n}}  \times \vec{n}\right|_{\Gg_s}  =0,
\label{TEQ-3} \\
&\left. { \left( \pp_z \widetilde{T}+ \frac{\aa}{K_v}  \widetilde{T}
 \right)} \right|_{z=0}= 0; \;
\left. \pp_z \widetilde{T}  \right|_{z=-H}= 0;
\; \left. \frac{\pp}{\pp \vec{n}} \widetilde{T}  \right|_{\Gg_s}= 0,
 \label{TEQ-4} \\
&\widetilde{T} (x,y,z,0) = \widetilde{T}_0= T_0 (x,y,z)-T^*(x,y),
\label{TEQ-5}
\end{align}
where
\begin{equation}
Q^* = Q +  K_h \Dd  T^*.
\end{equation}
\end{subequations}
In the above system the unknowns are the vector field
$u(x,y,z,t)$, and the scalar functions $p_s(x,y,t)$ and
$\widetilde{T} (x,y,z,t)$; while $T^*, \tau, Q^*$ and
$\widetilde{T}_0$ are given.

It is clear that once we determine $u(x,y,z,t), p_s(x,y,t)$ and
$\widetilde{T} (x,y,z,t)$ we can easily recover, thanks to
(\ref{WWWW}), (\ref{PPPP}) and (\ref{TTTT}), the original unknowns of
system (\ref{EQ-1})--(\ref{INIT}), i.e., $(u,w), T$ and $p$,
which makes the new formulation equivalent to the original system
(\ref{EQ-1})--(\ref{INIT}).  For the mathematical theory and global well-posedness of this model we direct the reader to \cite{PG_Cao_Titi, STW98}.

The purpose of this article is to introduce and analyse a data assimilation (downscaling) algorithm for recovering the solution $u$ and $T$ of system \eqref{PG1} from coarse spatial measurements of the temperature $T$ alone, in the absence of intial condition $T_0$.  Equivalently, we consider system \eqref{PG}.  Suppose that the coarse measurement of $T$ or equivalently of $\widetilde{T}$ are represented by the interpolant operator $I_h(\widetilde{T}), where $h$ is the size of coarse spatial mesh of the observation $ (see, e.g. \cite{Azouani_Olson_Titi} for details).  We propose the following algorithm for the approximate velocity $v$, temperature $\eta$ and pressure $q_s$, that are approximating the unknown reference solution $T$, $u$ and $p_s$, respectively:
\begin{subequations}\label{DA_PG}
\begin{align}
&\nabla  \left[ q_s(x,y,t) - \int_{-H}^z \eta  (x,y,\xi,t) d\xi
-  (z+H) T^* (x,y,t) \right]+ f \vec{k} \times v+  L_1 v  = 0,  \label{DA_TEQ-1}   \\
&\nabla \cdot \int_{-H}^0  v(x,y,z,t) \; dz = 0,    \label{DA_TEQ-11}     \\
&\pp_t \eta  +L_2 \eta
+ v \cdot \nabla \eta
-\left( \nabla \cdot \int_{-H}^z  v(x,y,\xi,t) \,d\xi
\right)  \pp_z \eta   +   v  \cdot \nabla T^*  \notag \\
& \hspace{2.8in}=Q^* - \mu \left(I_h(\eta) - I_h(\widetilde{T})\right),  \label{DA_TEQ-2}   \\
&\left. \frac{\pp v }{\pp z} \right|_{z=0} = \tau, \;
\left. \frac{\pp v }{\pp z} \right|_{z=-H} = 0, \;
\left. v\cdot \vec{n} \right|_{\Gg_s} = 0, \;
\left. \frac{\pp v}{\pp \vec{n}}  \times \vec{n}\right|_{\Gg_s}  =0,
\label{DA_TEQ-3} \\
&\left. { \left( \pp_z \eta + \frac{\aa}{K_v}  \eta
 \right)} \right|_{z=0}= 0; \;
\left. \pp_z \eta   \right|_{z=-H}= 0;
\; \left. \frac{\pp}{\pp \vec{n}} \eta   \right|_{\Gg_s}= 0,
 \label{DA_TEQ-4} \\
&\eta  (x,y,z,0) =  \eta_0,
\label{DA_TEQ-5}
\end{align}
where
\begin{equation}
Q^* = Q +  K_h \Dd  T^*.
\end{equation}
\end{subequations}
The unknowns are the vector field
$v(x,y,z,t)$, and the scalar functions $q_s(x,y,t)$ and
$\eta (x,y,z,t)$; while $T^*$ and $ \tau, Q^*$are given.
Here, $\eta_0$ can be taken arbitrary and $I_h(\cdot)$ is a linear interpolant operator based on the observational measurements  on a coarse spatial resolution of size $h$, for $t\in [0,T]$. Let us denote by $L^2 (\Om)$ and $H^1(\Om), H^2(\Om), \cdots,$
the usual $L^2-$Lebesgue and Sobolev spaces, respectively (see, e.g.  \cite{Evans} and \cite{ZW89}). Two types of interpolants can be considered.
One is to be given by a linear interpolant operator $I_h: H^{1}(\Omega) \rightarrow L^{2}(\Omega)$ satisfying
 the approximation property
\begin{align}\label{app}
\norm{\varphi - I_h(\varphi)}_{L^2(\Omega)}^2 \leq c_0h^2\norm{\varphi}_{H^1(\Omega)}^2,
\end{align}
for every $\varphi \in H^{1}(\Omega)$, where $c_0>0$ is a dimensionless constant.
The other type is given by $I_h: H^{2}(\Omega)\rightarrow L^{2}(\Omega)$, together with
\begin{align}\label{app2}
\norm{\varphi - I_h(\varphi)}_{L^{2}(\Omega)}^2 \leq c_0h^2\norm{\varphi}_{H^{1}(\Omega)} + c_0^2h^4\norm{\varphi}_{H^{2}(\Omega)}^2,
\end{align}
for every $\varphi \in H^{2}(\Omega)$, where $c_0>0$ is a dimensionless constant.

To give an example of an interpolant operator that satisfies \eqref{app}, we consider the positive definite, self-adjoint Laplace operator $(-\Delta)$ for the temperature with the corresponding boundary condition  \eqref{DA_TEQ-4}.  This linear operator has a compact inverse $(-\Delta)^{-1}:L^2(\Om)\rightarrow L^2(\Om)$, thus there exist a complete orthonormal set of eigenfunctions $\{w_j\}_{j=1}^\infty \subset L^2(\Om)$ such that $-\Delta w_j = \lambda_j w_j$, where $0 < \lambda_j \leq \lambda_{j+1}$ for $j\in \mathbb{N}$.    Since we can order the eigenvalues  we can let $I_h$ to be the orthogonal projection of $L^2(\Omega)$ onto the linear subspace spanned by the first $m_h$ eigenfunctions $\{w_1, w_2, \dots w_{m_h}\}$, where $m_h$ is chosen large enough so that the corresponding eigenvalue $(\lambda_{m_h})^{-1} \leq h^{-2}$.  In the case of periodic boundary conditions, an example of an interpolant observable that satisfies \eqref{app}, is the orthogonal projection onto the linear subspace spanned by the low Fourier modes with wave numbers $k$ such that $|k|\leq m_h = 1/h$.
   Physically relevant example is based on volume elements measurements that was studied in {\cite{Azouani_Olson_Titi, Jones_Titi}}. Examples of an interpolant observable that satisfies \eqref{app2} are given by the low Fourier modes and the measurements at a discrete set of nodal points in $\Omega$ (see Appendix A in {\cite{Azouani_Olson_Titi})}. We are not treating the second type of interpolants in this paper only for the simplicity of presentation. For full details on the analysis for the second type of interpolants we refer to { \cite{Azouani_Olson_Titi, FJT, FLT2015, FLT_horizontal}}.

The aim of this paper is to analyse system \eqref{DA_PG} and to show that its solutions converge, at an exponential rate in time, to the unknown corresponsing reference solution $\widetilde{T}$, $u$ and $p_s$ of \eqref{PG}.  It is worth mentioning that by combining the tools developed in this paper with those in \cite{Bessaih_Olson_Titi} we can treat in a similar way the case when the measurements are contaminated with a noisy stochastic error.  Furthermore, employing  the ideas in \cite{Foias_Mondaini_Titi} with the tools developed here one can treat in a similar fashion the case when we have fully discrete measurements in space and time.  That is, the case where the coarse spatial mesh measurements are collected at discrete times, $\left\{t_j\right\}_{j=1}^\infty$, provided $|t_{j+1 - t_j}|\leq \kappa$, for $\kappa$ small enough depending on the physical parameters.  We avoid here the treatment of the most general case in order to simplify the presentation.  However, the combination of the ideas from this work and those of \cite{Foias_Mondaini_Titi} is almost straightforward, yet tedious.

\bigskip
\section{Preliminaries and Functional Setting}  \label{S-1}

\subsection{Functional spaces and  relevant inequalities}

We denote by
\begin{equation}
| T | = \left(  \int_{\Om}  |T  (x,y,z)|^2 \; dxdydz
\right)^{\frac{1}{2}},   \label{L2}
\end{equation}
for every $T\in L^2(\Om)$,  and by
\begin{eqnarray}
&&\| T \| = \left( \aa \int_{\Gg_u} | T(x,y,0)|^2 dxdy+
 \int_{\Om} {\left[K_h |\nabla T (x,y,z)|^2 + \right.}\right. \label{NORM}
 \\
&&\hskip0.5in
 +\left. { \left.
{ K_v \left| \pp_z T (x,y,z) \right|^2} \right] }\; dxdydz
\right)^{\frac{1}{2}},   \nonumber
\end{eqnarray}
for every $T \in H^1 (\Om)$. Let
\begin{eqnarray*}
\widetilde{\mathcal{V}} &=&  \left\{ { \widetilde{T} \in
C^{\infty}(\overline{\Om}):
\left. \frac{\pp \widetilde{T}}{\pp z }
\right|_{z=-H}= 0;
\left. {\left( \frac{\pp \widetilde{T}}{\pp z}
+ \frac{\aa}{K_v} \widetilde{T} \right) }
\right|_{z= 0}= 0;
\; \left. \frac{\pp \widetilde{T}}{\pp \vec{n}} \right|_{\Gg_s}= 0
} \right\}.
\end{eqnarray*}
We also denote by $H^{\prime}$  the dual space of $H^1(\Om)$,
with the dual action $\ang{\cdot,  \cdot}$.

Next, we recall the following Poincar\'{e}-type inequalities (cf., e.g., { \cite{AR75}, \cite{Evans} \cite{ZW89}})
\begin{proposition} \label{T-EQU}
The norm defined as in {\em (\ref{NORM})} is equivalent to the
$H^1(\Om)$ norm.
That is, there is a positive constant $K_1$ such that
\begin{equation}
\frac{1}{K_1}
\| T \|^2  \leq  \| T \|^2_{H^1(\Om)}
\leq K_1 \| T \|^2
\label{TTT}
\end{equation}
for every $T \in H^1(\Om).$ Moreover, we have
\begin{eqnarray}
&&|T|^2 \leq \widetilde{K} \|T\|^2,
\qquad \qquad \mbox{for all} \quad T \in H^1(\Om),
\label{PIE}
\end{eqnarray}
where
\begin{eqnarray}
&& \widetilde{K} = \max \left\{ \frac{2H}{\aa},  \frac{2H^2}{K_v}
\right\}.          \label{K-T}
\end{eqnarray}

\end{proposition}

For convenience we state the following version of
Sobolev and Ladyzhenskaya interpolation inequalities (cf., e.g., {
\cite{AR75, LADY}}):
\begin{eqnarray}
&&\left\{ \begin{array}{l}
\displaystyle{\| \phi(x,y) \|_{L^4(M)}  \leq C_4  \| \phi(x,y)
\|_{L^2(M)}^{1/2}
\| \phi(x,y) \|_{H^1(M)}^{1/2},}  \\
\displaystyle{\| \phi(x,y) \|_{L^6(M)}  \leq C_4  \| \phi(x,y)
\|_{L^2(M)}^{1/3}
\| \phi(x,y) \|_{H^1(M)}^{2/3},}
\end{array} \right.  \label{SIT-2}
\end{eqnarray}
and
\begin{eqnarray}
&&
\left\{ \begin{array}{l}
\displaystyle{\| g(x,y,z) \|_{L^3(\Om)}  \leq C_5
| g(x,y,z) |^{1/2} \| g(x,y,z) \|_{H^1(\Om)}^{1/2},}  \\
\displaystyle{\| g(x,y,z) \|_{L^6(\Om)}  \leq C_5
\| g(x,y,z) \|_{H^1(\Om)},}
\end{array} \right.   \label{SIT-3}
\end{eqnarray}
for all $ \phi \in H^1(M)$ and $g \in H^1(\Om)$, respectively.
Also, we recall the integral version of Minkowsky inequality for the $L^p$
spaces, $p\geq 1$. Let $\Om_1 \subset \mathbb{R}^{m_1}$ and
 $\Om_2 \subset \mathbb{R}^{m_2}$ be two measurable sets, where
$m_1$ and $m_2$ are two positive integers. Suppose that
$f(\xi,\eta)$ is measurable over $\Om_1 \times \Om_2$. Then,
\begin{equation}
\hskip0.35in \left[ { \int_{\Om_1} \left( \int_{\Om_2} |f(\xi,\eta)| d\eta
\right)^p d\xi } \right]^{1/p}
\leq \int_{\Om_2} \left( \int_{\Om_1} |f(\xi,\eta)|^p d\xi
\right)^{1/p} d\eta.
\label{MKY}
\end{equation}
Hereafter, $C,$ which may depend on the domain $\Om$ and the constant
parameters $ f_0, \beta, \aa, A_h, A_v, K_h, K_v$
in the system (\ref{EQ-1})--(\ref{INIT}),
will denote a constant that may change from
line to line.

We will apply the following inequality which is a particular case of a more general inequality proved in \cite{Jones_Titi}.
\begin{lemma}\label{gen_gron_2}\cite{Jones_Titi} Let $\tau>0$ be fixed. Suppose that $Y(t)$ is an absolutely continuous nonnegative function which is locally integrable and that it satisfies the following:
\begin{align*}
\od{Y}{t} + \alpha(t) Y \leq \beta(t),\qquad \text{ a.e. on } (0,\infty),
\end{align*}
such that
\begin{align}\label{cond_1}
\liminf_{t\rightarrow\infty} \int_t^{t+\tau} \alpha(s)\,ds \geq \gamma, \qquad 
\limsup_{t\rightarrow\infty} \int_t^{t+\tau}  \alpha^{-}(s)\,ds < \infty,
\end{align}
and
\begin{align}
\lim_{t\rightarrow \infty} \int_t^{t+\tau} \beta^{+}(s)\,ds = 0,
\end{align}
for some $\gamma>0$, where $\alpha^{-} = \max\{-\alpha, 0\}$ and $\beta^{+} = \max\{\beta,0\}$.
Then, $Y(t)\rightarrow 0$ at an exponential rate, as $t\rightarrow \infty$.
\end{lemma}

Finally, we state the following proposition that was proved in { \cite{PG_Cao_Titi}}.
\begin{proposition}\cite{PG_Cao_Titi}\label{PROP}
Let $u=(u_1, u_2)\in H^2(\Omega)$, $f\in L^2(\Omega)$ and $g\in H^1(\Omega)$. Then
\begin{eqnarray*}
&& \left|{  \int_{\Om}
     \left(  \nabla \cdot \int_{-H}^z  u(x,y,\xi,t) \,d\xi
     \right)  f(x,y,z)\; g(x,y,z) \; dxdydz
  }\right|  \\
&&\hskip.06in
\leq C \left| f \right| \|u\|_{H^1(\Om)}^{1/2} \|u\|_{H^2(\Om)}^{1/2}
\left\| g  \right\|_{H^1(\Om)}^{1/2}
\left| g \right|^{1/2}.
\end{eqnarray*}
\end{proposition}

\subsection{Regularity Results}

We state the definition of weak solutions.

\begin{definition} \cite{PG_Cao_Titi}\label{D-1}
\thinspace Let $\widetilde{T}_0 \in L^2 (\Om)$ and let
$S$ be any fixed positive time. \thinspace
The vector field $v(x,y,z,t)$, and the scalar functions
$p_s(x,y,t)$  and $\widetilde{T}(x,y,z,t)$ are called
a weak solution of {\em (\ref{TEQ-1})--(\ref{TEQ-5})} on
the time interval $[0,S]$ if
\begin{eqnarray*}
&&p_s (x,y,t)  \in C([0,S], L^2(M)) \cap L^2([0,S], H^1 (M)), \\
&&u(x,y,z,t)  \in C([0,S],H^1(\Om)) \cap L^2([0,S], H^2 (\Om)), \\
&&\widetilde{T} (x,y,z,t)  \in C([0,S],L^2 (\Om) ) \cap L^2([0,S],
H^1(\Om)),     \\
&&\pp_t \widetilde{T} (x,y,z,t) \in
L^1([0,S], H^{\prime}),
\end{eqnarray*}
(recall that $H^{\prime}$ is the dual space of $H^1(\Om)$),
and if they satisfy
\begin{align*}
&
\int_{\Om} \nabla  \left[ p_s(x,y,t)
- \int_{-H}^z (\widetilde{T} (x,y,\xi,t)+T^*) d\xi
\right] \phi \, dxdydz+ \\
& + \int_{\Om} \left( f \vec{k} \times u\right) \phi  \, dxdydz
 +  \int_{\Om} \left(
A_h \nabla u \cdot \nabla \phi + A_v \pp_z u \pp_z \phi \right)  \,
dxdydz  \\
& = \int_{\Gg_u} A_v \tau \phi  \, dxdydz,   \nonumber
\end{align*}
and
\begin{align*}
&
\int_{\Om} \widetilde{T}(t) \psi \, dxdydz  +\int_{t_0}^t \int_{\Om}
{\left(K_h \nabla \widetilde{T} \cdot \nabla
  \psi + K_v \pp_z \widetilde{T} \pp_z \psi\right)}  \, dxdydz
  \;\;  \\
&+\aa \int_{t_0}^t \int_{\Gg_u} \widetilde{T} \psi \, dx dy
+\int_{t_0}^t \int_{\Om} \left( u  \cdot \nabla T^* \right) \psi \,dxdydz
+   \nonumber   \\
&+\int_{t_0}^t \int_{\Om}  \left[ { \left( u
 \cdot \nabla \widetilde{T} \right) } \; \psi
-\left( \nabla \cdot \int_{-H}^z  u(x,y,\xi,t) \,d\xi
\right)  \pp_z \widetilde{T} \; \psi  \right]  \,
dxdydz  \\
& =\int_{\Om} \widetilde{T}(t_0)\psi  \, dxdydz+ \int_{t_0}^t \int_{\Om} Q^* \psi \, dxdydz,
\end{align*}
for every $\phi \in (C^{\infty}(\overline{\Om}))^2$
and $\psi \in C^{\infty}(\overline{\Om}),$
and for almost every $t$, $t_0\in [0,S]$.

Moreover, if $\widetilde{T}_0 \in H^1(\Om)$ a weak solution is called strong solution of
{\em (\ref{TEQ-1})--(\ref{TEQ-5})} on $[0,S]$ if, in addition, it
satisfies
\begin{eqnarray*}
&&p_s(x,y,t)  \in C([0,S],H^1(M)) \cap L^2([0,S], H^2 (M)),  \\
&&u(x,y,z,t)  \in C([0,S],H^1(\Om)) \cap L^2([0,S], H^2 (\Om)),  \\
&&\widetilde{T} (x,y,z,t)  \in C([0,S],H^1(\Om)) \cap L^2([0,S], H^2
(\Om)).
\end{eqnarray*}
\end{definition}

Now we recall the global existence and uniqueness results proved in {{\cite{PG_Cao_Titi}}}.

\begin{theorem}[Weak solutions]\cite{PG_Cao_Titi} \label{T-WEAK}
Suppose that $\tau \in H_0^1(M), T^* \in H^2(M)$ and $Q \in L^2(\Om).$
Then for every $\widetilde{T}_0 = T_0 -T^* \in L^2(\Om)$ and
$S>0,$ there is a unique weak solution  $(p_s,  v, \widetilde{T})$
($p_s$ is unique up to a constant) of
the system {\em (\ref{TEQ-1})--(\ref{TEQ-5})} on the interval $[0,S]$.

Furthermore, the weak solution of the system (\ref{TEQ-1})--(\ref{TEQ-5})
depends continuously on the initial data. That is, the problem is globally
well--posed.
\end{theorem}

\begin{theorem}[Strong solutions]\cite{PG_Cao_Titi}\label{T-MAIN}
Suppose that $\tau \in H_0^1(M), Q \in H^1(\Om)$ and $T^* \in H^2(M).$
Then
for every $\widetilde{T}_0 = T_0 -T^* \in H^1(\Om),$ and
$S>0,$ there is a unique
strong solution $\widetilde{T}$ of system
{\em (\ref{TEQ-1})--(\ref{TEQ-5})}.
\end{theorem}

\begin{theorem} [Global attractor] \cite{PG_Cao_Titi} \label{global_attractor}
Suppose that $\tau \in H_0^1(M), Q \in L^2 (\Om)$ and $T^* \in H^2(M).$
Then, there is a finite-dimensional  global compact attractor $\mathcal{A} \subset L^2(\Om)$
for the system {\em (\ref{TEQ-1})--(\ref{TEQ-5})}.  Moreover, when $t$ is large enough we have
\begin{subequations}
\begin{align}
&|\widetilde{T} (t)|^2  \leq  R_a(T^*, Q)
:=   2\widetilde{R}_a(T^*, Q) +2 \| T^*\|_{L^2(M)}^2, \label{R-A}  \\
& \int_{t}^{t+r} \|T(s) \|^2 \; ds  \leq
 K_r (r, Q, T^*), \\
& \| \widetilde{T}(t) \| \leq R_v (r,T^*,Q,\tau),
\end{align}
\end{subequations}
where
\begin{align}
 &\widetilde{R}_a(T^*, Q) : = 4 \aa \widetilde{K} \|T^*\|_{L^2 (M)}^2 +  8 \widetilde{K}^2 |Q|^2, \label{R_tilde}\\
 & K_r (r, Q, T^*) := 2 R_a(T^*, Q)+ \left[ 4 \aa \widetilde{K} \|T^*\|_{L^2 (M)}^2 + 8 \widetilde{K}^2 |Q|^2
\right] \; r,  \\
&R_v (r,T^*,Q,\tau) := C \left[ \frac{R_a(T^*, Q) }{r^{1/2}} + \| T^*\|_{H^1 (M)}+
|Q|  \right. \\
&
 \qquad\qquad  \left. + \frac{C}{\ll_1^{1/2}}
{\left( 1+ \|T^*\|_{H^2(\Om)}^2 +
|Q| + \| \tau \|_{H^1(M)}^2 + R^2_a(T^*, Q)
\right) } \right] \times  \nonumber   \\
&
\qquad\qquad  \times
e^{\displaystyle{ C \left[ (R_a(T^*, Q))^4
+ { \left(  \|T^*\|_{H^2(M)}^4 + \| \tau \|_{H^1(M)}^4
+ (R_a(T^*, Q))^4 \right)} \, r
\right]  }}.   \nonumber
\end{align}
\end{theorem}

\bigskip
\section{Analysis and Convergence of the Data Assimilation Algorithm} \label{convergence}

In this section, we derive conditions under which the solution $(q_s,v,\eta)$, of the data assimilation algorithm system (\ref{DA_TEQ-1})--(\ref{DA_TEQ-5}),  converges to the corresponding unique reference solution $(p_s, u,\widetilde{T})$ of the planetary geostrophic system (\ref{TEQ-1})--(\ref{TEQ-5}), at an exponential rate, as $t\rightarrow \infty$.

{\bf Remark:} The steps of the following proof are formal in the sense
that they can be made rigorous by proving their corresponding
counterpart estimates first for the Galerkin approximation system. Then the estimates for the exact
solution can be established by passing to the limit in the Galerkin
procedure by using the appropriate ``Compactness Theorems''.

\begin{theorem}\label{th_conv_1}
Suppose that $I_h$ satisfies the approximation property \eqref{app}. Let $(p_s(t), u(t),\widetilde{T}(t))$, for $t\geq 0$, be a strong solution in the global attractor of  \eqref{TEQ-1}--\eqref{TEQ-5}. Let $\eta_0 \in L^2(\Omega)$ and suppose that $\mu>0$ is large enough such that
\begin{align}\label{mu_1}
\mu\geq 2C\left(1 + 5 \widetilde{R}_a(T^*,Q) + 4 \norm{T^*}_{L^2(M)}^2 + \norm{T^*}_{H^1(M)}^{4/3}\right),
\end{align}
where $\widetilde{R}_a(T^*,Q)$ is a constant defined in \eqref{R_tilde}, and $h>0$ is small enough such that $\mu c_0^2 h^2\leq 1$. Then, for any $S>0$, system \eqref{DA_TEQ-1}--\eqref{DA_TEQ-5} has a unique weak solution $(q_s, v,\eta)$ on the time interval $[0,S]$ ($q_s$ is unique up to a constant, i.e., $\nabla q_s$ is unique) in the sense of Definition \ref{D-1}.

Moreover, the solution $(v,\eta)$ depends continuously on the initial data, and it satisfies $\norm{\eta(t)-\widetilde{T}(t)}_{L^2(\Omega)}^2 \rightarrow 0,$ and $\norm{v(t)-u(t)}_{H^{1}(\Omega)}^2 \rightarrow 0$, at an
exponential rate, as $t \rightarrow \infty$.
\end{theorem}

\begin{proof} We consider the difference between the reference solution and the approximate solution,  (\ref{TEQ-1})--(\ref{TEQ-5}) and (\ref{DA_TEQ-1})--(\ref{DA_TEQ-5}), respectively.  We provide here  the relevant {\it a priori} estimates to show simultaneously the global well-posedness and the convergence results.
Denote by
$U =v - u$,
$\chi =\eta-\widetilde{T},$ and
$P_s = q_s - p_s.$
Then,  $P_s$, $U$ and $\chi$ satisfy:
\begin{subequations}
\begin{align}
&
\nabla  \left[ P_s(x,y,t)
 - \int_{-H}^z \chi (x,y,\xi,t) d\xi
\right] + f \vec{k} \times U + L_1 U = 0,
  \label{WTEQ-1}  \\
&
\nabla \cdot \int_{-H}^0 U(x,y,z,t) \; dz = 0,    \label{WTEQ-11}     \\
&
\pp_t \chi  + L_2 \chi + U
 \cdot \nabla \widetilde{T} + v \cdot \nabla \chi
+   U  \cdot \nabla T^*   -\left( \nabla \cdot \int_{-H}^z  U(x,y,\xi,t) \,d\xi  
\right)  \frac{\pp \widetilde{T}}{\pp z} - \nonumber \\
&
\qquad \qquad \qquad \qquad \qquad  -\left( \nabla \cdot \int_{-H}^z  v(x,y,\xi,t) \,d\xi
\right)  \frac{\pp \chi}{\pp z}  - \mu I_h(\chi)=0,
   \label{WTEQ-3}   \\
&
\left. \frac{\pp U }{\pp z} \right|_{z=0} = 0, \;
\left. \frac{\pp U }{\pp z} \right|_{z=-H} = 0, \;
\left. U \cdot \vec{n} \right|_{\Gg_s} = 0, \;
\left. \frac{\pp U}{\pp \vec{n}}  \times \vec{n}\right|_{\Gg_s}  =0,
\label{WTEQ-4} \\
&
\left. {\left( \frac{\pp \chi}{\pp z} + \frac{\aa}{K_v}  \chi
 \right)} \right|_{z=0}= 0; \;
\left. \frac{\pp \chi}{\pp z}  \right|_{z=-H}= 0;
\; \left. \frac{\pp \chi}{\pp \vec{n}}  \right|_{\pp M}= 0,
 \label{WTEQ-5} \\
&
\chi (x,y,z,0) =\eta_0 (x,y,z)- \widetilde{T}_0 (x,y,z).
\label{WTEQ-6}
\end{align}
\end{subequations}
Next, we follow the ideas and arguments in \cite{PG_Cao_Titi}.   By averaging (\ref{WTEQ-1}) and (\ref{WTEQ-3})
with respect to $z$ and using (\ref{WTEQ-11}), we get
\begin{subequations}
\begin{align}
&
\nabla \left[ P_s(x,y,t) + \frac{1}{H} \int_{-H}^0
\xi \chi (x,y,\xi,t) \; d\xi \right]
+ f \vec{k} \times \overline{U} -  A_h \Dd \overline{U}
=0,       \label{WQ-1}   \\
&
\nabla \cdot  \overline{U} = 0,
     \label{WQ-2}  \\
&
\overline{U} \cdot \vec{n} = 0, \quad
\frac{\pp \overline{U}}{\pp \vec{n}} \times \vec{n} =0,
\qquad \qquad \mbox{ on } \pp M,
 \label{WQ-3}
\end{align}
\end{subequations}
where for any integrable function $\phi$ on $\Omega$ we denote by
\[
\overline{\phi} (x,y,t) = \frac{1}{H} \int_{-H}^0 \phi(x,y,z,t) \; dz.
\]
By taking the $L^2(\Om)$ inner product to equation (\ref{WQ-1}) with
$\overline{U}$, we obtain
\[
\int_{\Om} \left[ { \nabla \left( P_s(x,y,t) + \frac{1}{H} \int_{-H}^0
\xi \chi (x,y,\xi,t) \; d\xi \right)
-  A_h \Dd \overline{U} \,
} \right] \; \overline{U} \; dxdydz =0.
\]
By using integration by parts and applying (\ref{WQ-2}) and
(\ref{WQ-3}), we get
\[
\int_{\Om} |\nabla \overline{U}|^2 \; dxdydz =0.
\]
Thus, $\overline{U}$ is a function of $t$ alone. By (\ref{WQ-3}), we reach
$\overline{U} =0.$ As a result, we have
\begin{eqnarray}
&&\hskip-.68in
P_s (x,y,t) = - \frac{1}{H} \int_{-H}^0
\xi \, \chi (x,y,\xi,t) d\xi.    \label{SL-2}
\end{eqnarray}
($P_s$ is unique up to a constant that depends on time, thus $\nabla P_s$ is unique).
Therefore, (\ref{WTEQ-1}) can be written as
\begin{eqnarray}\label{eq_rew}
&&
- \nabla  \left[ \frac{1}{H} \int_{-H}^0
\xi \,  \chi (x,y,\xi,t) d\xi
 + \int_{-H}^z \chi (x,y,\xi,t) d\xi
\right] +  \nonumber \\
&& + f \vec{k} \times U + L_1 U = 0.
 \label{WTEQ-111}
\end{eqnarray}
Notice that $U$ satisfies the boundary condition (\ref{WTEQ-4}).
For the second order elliptic boundary--value problem \eqref{eq_rew}
we have the following regularity results (by following similar
techniques to those developed in {\cite{Hu-Temam-Ziane}} and
{\cite{Ziane}}. For the case of smooth domains see,
  for example, { \cite{LADY}} p. 89,
 and {\cite{US75}})
\begin{equation}
\|U\|_{H^1(\Om)} \leq \frac{C_2}{ \widetilde{A}} |\chi|,
\qquad \mbox{ and } \quad
\|U\|_{H^2(\Om)} \leq  \frac{C_2}{ \widetilde{A}} \|\chi\|,
\label{UR}
\end{equation}
where $\widetilde{A} = \min\{A_h, A_v\}$.
By taking the $H^{\prime}$ dual action of
equation (\ref{WTEQ-3}) with $\chi$, we obtain
\begin{align*}
&
\ang{\pp_t \chi  + L_2 \chi, \chi }
+ \ang{ U \cdot \nabla \widetilde{T} + v \cdot \nabla
\chi
+  U  \cdot \nabla T^* , \chi } -  \\
&
- \ang{( \nabla \cdot \int_{-H}^z  U(x,y,\xi,t) \,d\xi  )
\pp_z \widetilde{T} -
 ( \nabla \cdot \int_{-H}^z  v(x,y,\xi,t) \,d\xi
)  \pp_z \chi, \chi} -\mu\ang{I_h(\chi), \chi}=0.
\end{align*}

Notice that by integrating by parts and  using the
boundary conditions  (\ref{WTEQ-5}),  we have
\begin{align}
&\int_{\Om} \chi L_2 \chi
\; dxdydz = - \int_{\Om} \chi \left(
K_h \Dd \chi + K_v \pp^2_z \chi
\right) \; dx dydz      \nonumber \\
&=  \int_{\Om} \left[ K_h |\nabla \chi|^2 +
K_v |\pp_z \chi|^2 \right] \, dxdydz
-\int_{\Gg_u} K_v \chi \pp_z \chi dx dy
 \nonumber   \\
&= \int_{\Om} \left[ K_h |\nabla \chi|^2 +
K_v |\pp_z \chi|^2 \right] \, dxdydz
+ \aa \int_{\Gg_u} |\chi|^2 dx dy  \nonumber  \\
&= \| \chi\|^2. \label{POS}
\end{align}

We use the facts that
\[
\ang{\pp_t \chi, \chi } = \frac{1}{2} \frac{d
  |\chi|^2}{dt} \quad  \; \;
\ang{L_2 \chi, \chi } = \|\chi \|^2, \quad \mbox{and} \; \; \mu\ang{I_h(\chi), \chi} = \mu(I_h(\chi), \chi).
\]
Moreover,
\begin{align*}
&
\ang{ U \cdot \nabla \widetilde{T} + v \cdot \nabla
\chi
+  U  \cdot \nabla T^* , \chi }  =
\int_{\Om} \left[{ U \cdot \nabla \widetilde{T}
+ v\cdot \nabla \chi
+   U  \cdot \nabla T^* }\right] \chi  \; dx dydz, \\
&
\ang{\left( \nabla \cdot \int_{-h}^z  U(x,y,\xi,t) \,d\xi  \right)
\pp_z \widetilde{T} +
 \left( \nabla \cdot \int_{-H}^z  v(x,y,\xi,t) \,d\xi
\right)  \pp_z \chi, \chi}  \; dx dydz \\
&
= \int_{\Om} \left[{  \left( \nabla \cdot \int_{-H}^z  U(x,y,\xi,t)
\,d\xi
\right)  \pp_z \widetilde{T} +
 \left( \nabla \cdot \int_{-H}^z  v(x,y,\xi,t) \,d\xi
\right)  \pp_z \chi }\right] \chi  \; dx dydz.
\end{align*}
 Therefore, we have
\begin{align*}
&
\frac{1}{2} \frac{d |\chi|^2}{dt}  + \| \chi \|^2  =
\int_{\Om} \left[{ - U
 \cdot \nabla \widetilde{T} - v\cdot \nabla \chi
-   U  \cdot \nabla T^* }\right. + \\
&
 \left.{ + \left( \nabla \cdot \int_{-H}^z  U(x,y,\xi,t) \,d\xi
\right)  \pp_z \widetilde{T} +
 \left( \nabla \cdot \int_{-H}^z  v(x,y,\xi,t) \,d\xi
\right)  \pp_z \chi }\right] \chi.
\end{align*}
Next, we estimate in the above equation term by term.
\begin{itemize}
\item[{\it (I.)}] By integrating by parts and (\ref{WTEQ-4}), we reach
\begin{equation}
\int_{\Om} \left[{  v\cdot \nabla \chi -
\left( \nabla \cdot \int_{-H}^z  v(x,y,\xi,t) \,d\xi
\right)  \pp_z \chi  }\right] \chi  \; dx dydz=0.
 \label{W-2}
\end{equation}
\item[{\it (II.)}]
\[
\left|{ \int_{\Om}  U \cdot \nabla \left( \widetilde{T} +T^* \right)
\chi  \; dx dydz} \right|  \leq  \| \widetilde{T} +T^*\|_{H^1(\Om)}\;
 \| U \|_{L^6(\Om)}
\|\chi \|_{L^3(\Om)}.
\]
Applying (\ref{UR}) and (\ref{SIT-3}), we have
\[
\| U \|_{L^6(\Om)} \leq \frac{C}{ \widetilde{A}}  | \chi |,
\quad \mbox{and} \; \;
\|\chi \|_{L^3(\Om)} \leq C  | \chi |^{1/2}
 \| \chi \|^{1/2}.
\]
Thus,
\begin{align}
&\left|{ \int_{\Om}  U \cdot \nabla \left( \widetilde{T} +T^* \right)
\chi  }\right|    \leq C \left[ \| \widetilde{T} \|+
\| T^* \|_{H^1(M)}\right] | \chi |^{\frac{3}{2}}
 \| \chi \|^{\frac{1}{2}}.    \label{W-3}
 \end{align}
\item[{\it (III.)}]~Applying Proposition \ref{PROP} by setting $u=U, f=\pp_z \widetilde{T}$ and $g=\chi $,
respectively,  we have
\begin{align*}
& \left|{  \int_{\Om}
     \left(  \nabla \cdot \int_{-H}^z  U(x,y,\xi,t) \,d\xi
     \right)  \pp_z \widetilde{T} \chi
  }\right|  \\
&
\leq C \left\| \widetilde{T} \right\|_{H^1(\Om)} \|U\|_{H^1(\Om)}^{1/2} \|U\|_{H^2(\Om)}^{1/2}
\left\| \chi  \right\|
\left| \chi \right|^{1/2}
\end{align*}
Applying (\ref{UR}) to the above estimate, we get
\begin{align}
& \left|{ \int_{\Om} \left( \nabla \cdot \int_{-H}^z  U(x,y,\xi,t) \,d\xi
\right)  \pp_z \widetilde{T} \; \chi   dxdydz }\right|
  \leq C \| \widetilde{T} \|
| \chi | \| \chi \|.  \label{W-4}
\end{align}
\item[{\it (IV.)}]~Finally, thanks to the assumptions $\mu c_0^2h^2\leq 1$, \eqref{app}, and Young inequality, we have
\begin{align}
-\mu(I_h(\chi), \chi) &= - \mu(I_h(\chi)-\chi, \chi) - \mu |\chi|^2 \notag \\
& \leq \mu c_0 h \norm{\chi} |\chi| - \mu |\chi|^2 \notag \\
&\leq \frac{\mu c_0^2 h^2}{2} \| \chi\|^2 - \frac{\mu}{2} | \chi|^2 \notag \\
& \leq  \frac{1}{2} \| \chi\|^2 - \frac{\mu}{2} | \chi|^2. \label{W-5}
\end{align}
\end{itemize}

Therefore, from the above estimates (\ref{W-2})--(\ref{W-5})
we get
\begin{align*}
&
\frac{1}{2} \frac{d |\chi|^2}{dt}  + \frac{1}{2}\|\chi\|^2
\leq C ( \| \widetilde{T} \|+
\| T^* \|_{H^1(M)} ) | \chi |^{\frac{3}{2}} \| \chi
\|^{\frac{1}{2}}+  C \| \widetilde{T} \|
 | \chi | \| \chi \| - \frac{\mu}{2} | \chi|^2.
\end{align*}
By Young inequality, we obtain
\begin{align*}
&
\frac{d |\chi|^2}{dt}  + \|\chi\|^2
\leq \left[C \left(1+ \| \widetilde{T} \|^2+
\| T^* \|_{H^1(M)}^{4/3} \right) -\mu \right] | \chi |^2.
\end{align*}

Recall that from Proposition \ref{global_attractor} we have
$$ \int_{t}^{t+1} \norm{\widetilde{T}(s)}^2 \, ds \leq 2R_a(T^*,Q) + \widetilde{R}_a(T^*,Q)= 5 \widetilde{R}_a(T^*,Q) + 4 \norm{T^*}_{L^2(M)}^2,$$
for any $t\geq 0$, $\widetilde{R}_a(T^*,Q)$ is a constant defined in \eqref{R_tilde}. Thus, applying Lemma \ref{gen_gron_2} with $\tau =1$, and using condition \eqref{mu_1}, we can conclude that condition \eqref{cond_1} is satisfied, therefore we have
\begin{align}
\norm{\eta(t) - \widetilde{T}(t)}_{L^2(\Omega)}^2 = |\chi(t)|^2 \rightarrow 0, 
\end{align}
at at exponential rate, as $t\rightarrow \infty$. The regularity result \eqref{UR} yields
$$\norm{v(t)-u(t)}_{H^1(\Omega)}^2 = \norm{U(t)}_{H^1(\Omega)}^2 \leq \frac{C_2^2}{{\widetilde{A}}^2} |\chi(t)|^2 \rightarrow 0, $$
at an exponential rate, as $t\rightarrow \infty$.

\end{proof}

\bigskip
\section*{Acknowledgements}

The work of A.F. is supported in part by NSF grant  DMS-1418911. The work of E.L. is supported by the ONR grants N0001416WX01475 and N0001416WX00796.  The work of  E.S.T.  is supported in part by the ONR grant N00014-15-1-2333 and the NSF grants DMS-1109640 and DMS-1109645.

\bigskip

\frenchspacing
\bibliographystyle{plain}





\end{document}